\documentclass[11pt]{article}

\usepackage{amsmath,amsthm,amssymb}
\usepackage{graphicx}
\usepackage{hyperref}
\usepackage[ansinew]{inputenc}
\usepackage[mathcal]{euscript}
 \usepackage{comment}
\usepackage{xcolor}
\usepackage[all]{xy}
\usepackage{blkarray}
\usepackage{epsf}
\usepackage{shuffle}
\usepackage{mathtools}
\usepackage{pst-all}
\usepackage{mathrsfs}
\usepackage{float}    
\newtheorem{theorem}{Theorem}[section]
\newtheorem{proposition}[theorem]{Proposition}
\newtheorem{lemma}[theorem]{Lemma}
\newtheorem{corollary}[theorem]{Corollary}

\theoremstyle{definition}
\newtheorem{definition}[theorem]{Definition}
\newtheorem{example}[theorem]{Example}

\newtheorem{convention}[theorem]{Convention}



\title{Polynomial and spectra factorization of graphs obtained by iteration the operad of generalized graph composition.}
\author{Jean Liendo}
\date{\today}

\begin{document}

\maketitle

\begin{abstract}
The generalized composition graph is used by Cardoso and some researchers for factorization the adjacency spectrum and Laplacian of a simple graph. Because the generalized composition graph is an example of a set-theoretic linear operad, this operation
can be iterated at more than one level, where the complex language of partition refinement
in the iteration is represented in terms of Schröder trees, which allows us
to generalize the factorization of the adjacency spectrum and Laplacian of a simple graph presented
by Cardoso in terms of Schröder trees and colorings over the edges of a graph. Cardoso's technique has been generalized by other authors for the universal adjacency matrix of a graph. This work also presents generalized factorizations in terms of Schröder trees and colorings on the edges of a graph for the universal adjacency spectrum, the characteristic polynomial of the universal adjacency matrix, and the generalized characteristic polynomial of a graph.
\end{abstract}

\vspace{0.5cm} 
\vspace{0.5cm}

\section{Preliminaries.
Fiedler's Lemma and its generalization as an application to the computation of the adjacency and Laplacian spectrum.}

In this work, all graphs are assumed to be finite and simple. If $g$ is a graph, we denote by $V(g)$ and $E(g)$ the sets of vertices and edges of $g$, respectively. Given a vertex $v \in V(g)$, the \emph{valency} $val(v)$ of $v$ is the number of edges of $g$ adjacent to $v$. A \emph{$k$-regular graph} is a graph $g$ such that $val(v)=k$, for every $v \in V(g)$. In this case, we say that $k=\mathrm{reg}(g)$ is the \emph{regularity} of $g$.

\medbreak

Suppose that the set of vertices of a graph $g$ is endowed with a total order. In this case, we write $V(g)=\{v_1, v_2, \hdots, v_n\}$, where $v_i < v_j$ if $i<j$. Three important $n\times n$ integer matrices can be associated to the graph $g$: its \emph{adjacency}, \emph{degree} and \emph{Laplacian} matrices $A(g)$, $D(g)$ and $L(g)$, respectively, defined as
\begin{equation}
A(g)_{i,j}:= \left\lbrace \begin{array}{cl}
		0,&\text{ if } \{v_i, v_j\}\not\in E(g)\\
		1,&\text{ if } \{v_i, v_j\}\in E(g)
	\end{array}\right. \quad , \quad
L(g)_{i,j}:=
\left\lbrace \begin{array}{cl}
		0,&\text{ if } \{v_i, v_j\}\not\in E(g)\\
		-1,&\text{ if } \{v_i, v_j\}\in E(g)\\
  val(v_i), &\text{ if }  v_i=v_j
	\end{array}\right.
\end{equation}
and $D(g)_{i,i}:=val(v_i)$, and 0 otherwise. By definition, $L(g)= D(g)-A(g)$, for every graph $g$.

\medbreak

If $A$ is a square matrix, $\sigma(A)$ will denote the \emph{spectrum} of $A$, which consists of the multiset formed by the eigenvalues of $A$. Given a graph $g$, the spectrum of its adjacency and Laplacian matrices of $g$ are denoted by $\sigma(A(g))$ and $\sigma(L(g))$, respectively.
\begin{convention}
The following convention will be used throughout this work. Given a set $X$, the \emph{free monoid} on $X$ (also called the \emph{Kleene star} on $X$) is the monoid whose elements are all the finite sequences of elements from $X$. Such finite sequences are called \emph{words} of the monoid. A \emph{letter} is a word with only one element. There is a unique word without elements, called \emph{empty word}. The monoidal operation in $X^*$ is the \emph{concatenation} of words, with the empty word as the unit for this operation. Now, if $\lambda_1, \lambda_2,..., \lambda_k$ are the distinct eigenvalues of a square matrix $A$ with respective multiplicities $m_1, m_2,..., m_k$, we will represent the multiset $\sigma(A)$ as a commutative word in the free monoid $\mathbb{R}^*$:
\begin{equation}
\sigma(A)=\lambda^{m_1}_1\lambda^{m_2}_2\cdots\lambda^{m_k}_k \in \mathbb{R}^*.
\end{equation}

\

Also, for every $1 \leq i \leq k$, we will write $\displaystyle \frac{\sigma(A)}{\lambda_i}$ for the word $\sigma(A)$ without one factor $\lambda_i$.
\end{convention}
\medbreak

\begin{example}
Consider the graph $g$ in the Figure \ref{exagraph}:
\begin{figure}[h]
    \centering
    \includegraphics[scale=0.4]{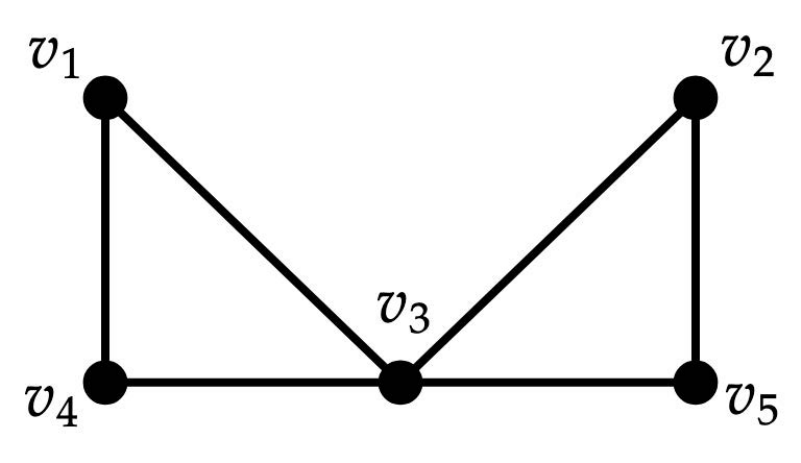}
    \caption{Graph $g$}
    \label{exagraph}
\end{figure}
from the corresponding adjacency and Laplacian matrices,
\[A(g)=
\begin{bmatrix}
0&0&1&1&0\\
0&0&1&0&1\\
1&1&0&1&1\\
1&0&1&0&0\\
0&1&1&0&0
\end{bmatrix}
\quad , \quad  L(g)=
\begin{bmatrix}
2&0&-1&-1&0\\
0&2&-1&0&-1\\
-1&-1&4&-1&-1\\
-1&0&-1&2&0\\
0&-1&-1&0&2
\end{bmatrix}\]
the resulting spectra are
\[\sigma(A(g))= {\big (}(1+\sqrt{17})/2{\big)}{\big (}(1-\sqrt{17})/2{\big)}(-1)^2\,1 \quad \text{ and } \quad \sigma(L(g)) = 0\, 1\, 3^2 \, 5.\]
In particular,
\[\frac{\sigma(A(g))}{(-1)}= {\big (}(1+\sqrt{17})/2{\big)}{\big (}(1-\sqrt{17})/2{\big)}(-1)\,1 \quad \text{ and } \quad \frac{\sigma(L(g))}{0} = 1\, 3^2 \, 5.\]
\end{example}

\subsection{Generalized Composition of Graphs}

Once a certain graph operation is defined, the next step is to study whether this operation decomposes the adjacency and Laplacian matrices into submatrices defined by the subgraphs that are part of the operation, and ultimately, to determine if there is a factorization of the adjacency or Laplacian spectrum, such is the case of \textit{the generalized composition of graphs} introduced by Schwenk in \cite{Schwenk} and reformulated by Cardoso in \cite{Cardoso2013}.

\
Let $\ell=(u_1, u_2,..., u_n)$ a linear order. We say that $\pi=(\ell_1, \ell_2,..., \ell_k)=(\ell_i)^k_{i=1}$ is a \emph{segmented partition or strong composition} of $\ell$ if each $\ell_j$ is a segment of $\ell$, i.e., $\ell_1={u_1, u_2,... u_{r_1}}$, $\ell_2={u_{r_1+1},..., u_{r_2}}$, ..., $\ell_k={u_{r_{k-1}+1},..., u_n}$. Also, $\mathit{a}=(g_{\ell_1}, g_{\ell_2},...,g_{\ell_k})=(g_{\ell_i})^k_{i=1}$ is a \emph{segmented or ordered assembly} of simple graphs over $\pi$, where for each $j=1,2,...,k$, $g_{\ell_j}$ is a simple graph whose vertex set is the segment $\ell_j$ of $\ell$. If $h$ is any simple graph whose vertex set is the segmented partition $\pi$ of $\ell$, the operation \textit{generalized composition of graphs}, $\displaystyle{\bigvee_h a}$ is defined as the graph $g$ whose vertex set is $\ell$ and edge set is given by:

\[\displaystyle{E(g)=\left(\bigcup_{j=1}^kE(g_{\ell_j})\right)\cup\left(\bigcup_{\{\ell_r, \ell_s\}\in E(H)}\{\{x,y\}:x\in \ell_r, y\in \ell_s\}\right)}\]

\

As illustrated in Figure \ref{f1}, $g$ is the result of connecting each vertex of $g_{\ell_r}$ with each vertex of $g_{\ell_s}$ whenever $\ell_r$ and $\ell_s$ are connected in $h$. In particular, if $k = 2$, we have the join operation $g_{\ell_1 }\vee g_{\ell_2}$. The generalized composition $g=\displaystyle{\bigvee_h a}$ further decomposes the graph $g$ into a pair $(a, h)$, where $a$ is interpreted as the internal ordered assembly of $g$ and $h$ is the external graph whose vertices are the components of $a$. In other words, the pair $(a, h)$ can be understood as the graph consisting of the assembly $a$ within the graph $h$. In practice, we draw the graph $h$ and replace the vertex $\ell_j$ in $h$ with the internal graph $g_{\ell_j}$. Finally, the generalized composition of graphs can be viewed functionally as follows:

\[(a, h) \longmapsto \displaystyle{\bigvee_h a}\]
\begin{figure}[h]
    \centering
    \includegraphics[width=0.9\textwidth]{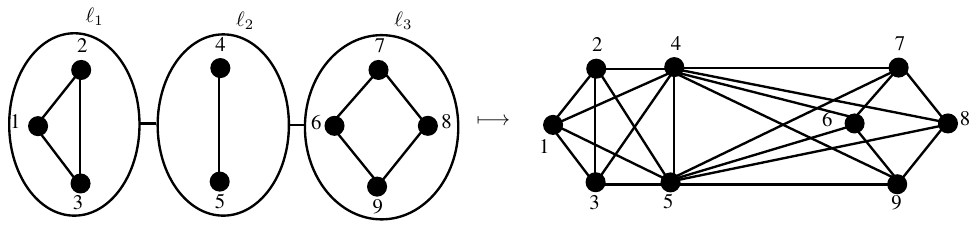}
    \caption{Example of generalized composition of graphs}
    \label{f1}
\end{figure}

The adjacency and Laplacian matrices of $g=\displaystyle{\bigvee_h a}$ are represented by the segments of $\pi$, the adjacency and Laplacian matrices of the graphs in $a$, and the edges of the graph $h$. In fact, we have
\begin{proposition}\label{Matriz join}
Let $a=(g_{\ell_j})^k_{j=1}$ be an assembly of simple graphs with segmented partition $\pi=(\ell_j)^k_{j=1}$ such that $|\ell_j|=n_j$. Let $h$ be a simple graph with vertex set $\pi$. For $j\in \{1, 2,..., k\}$, let $N_j=\displaystyle{\sum_{\{\ell_s,\ell_j\}\in E(h)} n_s}$ the external valence of $\ell_j$. If $g=\displaystyle{\bigvee_h a}$, then
\begin{enumerate}
\item The adjacency matrix $A(g)$ is represented over the segments of $\pi$ as follows:
    \begin{equation}\label{Adyacencia particionada}
    (A(g))_{\ell_i\ell_j}=\delta_{ij}A(g_{\ell_j}) + (A(h))_{ij}1_{n_i\times n_j}
    \end{equation}

\item The Laplacian matrix $L(g)$ is represented over the segments of $\pi$ as follows:

\begin{equation}\label{Laplaciana particionada}
    (L(g))_{\ell_i\ell_j}=\delta_{ij}(L(g_{\ell_j})+ N_jI_{n_j}) - (A(h))_{ij}1_{n_i\times n_j}
    \end{equation}
\end{enumerate}

In the above equations, $1_{n_i\times n_j}$ is the $n_i\times n_j$ matrix with all entries equal to 1, $I_{n_j}$ is the $n_j\times n_j$ identity matrix, and $\delta_{ij}$ is the Kronecker delta.
\end{proposition}

\

The equalities (\ref{Adyacencia particionada}) and (\ref{Laplaciana particionada}) can be written using the direct sum of matrices and the \textit{partitioned right Hadamard product}. This latter operation is defined as follows: If $A$ is a $k\times k$ matrix, $B$ is an $n\times n$ matrix with $k\leq n$, and $\pi=(\ell_j)^k_{j=1}$ is a segmented partition of $(1,2,...,n)$, \emph{the partitioned right Hadamard product}, $A\circ_{\pi}B$, is defined as the matrix represented by the segments of $\pi$.

\[(A\circ_{\pi}B)_{\ell_i\ell_j}=A_{ij}B_{\ell_i\ell_j}\]
Therefore,
\[A\left(\bigvee_h a\right)=\left(\bigoplus^k_{j=1} A(g_{\ell_j})\right)+ A(h) \circ_\pi 1_{n\times n}\]
\[L\left(\bigvee_h a\right)=\left(\bigoplus^k_{j=1} (L(g_{\ell_j})+N_jI_{n_j})\right)- A(h) \circ_\pi 1_{n\times n}\]

The adjacency and Laplacian matrices of the graph $g$ corresponding to Figure \ref{f1} are partitioned as follows:

\begin{eqnarray*}
A(g)&=&\begin{blockarray}{cccc}
 &\ell_1&\ell_2&\ell_3\\
\begin{block}{r(c|c|c)}
\ell_1& A(g_{l_1})& 1_{n_1 \times n_2} & 0_{n_1 \times n_3}\\
\cline{2-4}
\ell_2& 1_{n_2 \times n_1}& A(g_{l_2}) & 1_{n_2 \times n_3}\\
\cline{2-4}
\ell_3& 0_{n_3 \times n_1} & 1_{n_3\times n_2} & A(g_{l_3})\\
\end{block}
\end{blockarray}\\
L(g)&=&\begin{blockarray}{cccc}
 &\ell_1&\ell_2&\ell_3\\
\begin{block}{r(c|c|c)}
\ell_1& L(g_{\ell_1})+ N_1 I_{n_1}& -1_{n_1 \times n_2} & 0_{n_1 \times n_3}\\
\cline{2-4}
\ell_2& -1_{n_2 \times n_1}& L(g_{\ell_2})+ N_2 I_{n_2}& -1_{n_2 \times n_3}\\
\cline{2-4}
\ell_3& 0_{n_3 \times n_1} & -1_{n_3\times n_2} & L(g_{\ell_3})+ N_3 I_{n_1}\\
\end{block}
\end{blockarray}
\end{eqnarray*}

\

\subsection{Fiedler's lemma and spectra of generalized composition of graphs}
In his study for necessary and sufficient conditions for $n$ real numbers to be eigenvalues of a nonnegative $n \times n$ matrix, Miroslav Fiedler came across the following simple result:

\begin{lemma}[Fiedler, Lemma 2.2 in \cite{Fidler1974}]
Let $M_1, M_2$ be symmetric matrices of size $m \times m$ and $n \times n$, respectively. Let $\alpha_1, \hdots, \alpha_m$ and $\beta_1, \hdots, \beta_n$ the corresponding eigenvalues of $A$ and $B$. Suppose that $\alpha_1$ (resp. $\beta_1$) has a unit eigenvector $u$ (resp. $v$). Then for any real value $\rho$, the matrix
\begin{equation}\label{Fiedler}
C=
\begin{bmatrix}
 M_1 & \rho uv^t  \\
 \rho vu^t & M_2
\end{bmatrix}
\end{equation}
has eigenvalues $\alpha_2,..., \alpha_m$, $\beta_2,...,\beta_n$, $\gamma_1, \gamma_2$, where $\gamma_1, \gamma_2$ are eigenvalues of the matrix
\[\widehat{C}=
\begin{bmatrix}
\alpha_1 & \rho   \\
\rho  & \beta_1
\end{bmatrix}\]
\end{lemma}

\
We refer to this result as \emph{Fiedler's lemma}. This result allows Fiedler to prove that known sufficient conditions for $n$ real numbers to be eigenvalues of a nonnegative $n \times n$ matrix are also sufficient for the existence of a symmetric matrix with these eigenvalues (see Theorem 2.5 in \cite{Fidler1974}). Under our convention, Fiedler's lemma reads as follows:
\begin{equation}
\sigma(C) = \frac{\sigma(M_1)}{\alpha_1}\cdot\frac{\sigma(M_2)}{\beta_1}\cdot \sigma(\widehat{C}).
\end{equation}

\medbreak

From the work of Cardoso et al. in \cite{Cardoso2011}, \cite{Cardoso2013} and \cite{Cardoso2022}, Fiedler's lemma has been proved to be useful for calculating the eigenvalues of the generalized composition of graphs, under regularity assumptions. We first illustrate the method used in \cite{Cardoso2013} to compute the eigenvalues of the join operation of two regular graphs. Let $g_1$ and $g_2$ be regular graphs, with regularities $r_1$ and $r_2$, respectively and $m =|V(g_1)|, n=|V(g_2)|$. By definition, the adjacency matrix of $g_1 \vee g_2$ is a $mn$-square matrix containing $A(g_1)$ and $A(g_2)$ as diagonal matrices. Fiedler's lemma implies directly the following description of the eigenvalues of $g_1 \vee g_2$:
\begin{equation}\label{spAdj2}
\sigma(A(g_1\vee g_2))=\frac{\sigma(A(g_1))}{r_1}\cdot \frac{\sigma(A(g_2))}{r_2} \cdot \sigma \begin{bmatrix}
    r_1& \sqrt{mn}\\
    \sqrt{mn} & r_2
\end{bmatrix}.
\end{equation}

A similar description can be obtained from Fiedler's lemma when computing the spectrum of the Laplacian associated to $g_1 \vee g_2$:
\begin{equation}\label{spLp2}
\sigma(L(g_1\vee g_2))=\frac{\sigma(L(g_1) + n \text{Id}_m)}{n}\cdot \frac{\sigma(L(g_2) + m \text{Id}_n)}{m} \cdot \sigma \begin{bmatrix}
    n& -\sqrt{mn}\\
    -\sqrt{mn} & m
\end{bmatrix}.
\end{equation}

\

In order to give an analogue of formulas \eqref{spAdj2} and \eqref{spLp2} for the eigenvalues of the generalized composition of graphs, Cardoso et al. introduce the following generalization of Fiedler's lemma.

\begin{lemma}[Generalized Fiedler's Lemma (Cardoso et al., \cite{Cardoso2013})]\label{Fiedler Generalizado}
Let $\mathcal{M}=\{M_j\}^k_{j=1}$ be a set of symmetric matrices such that for each $j=1,2,...,k$, the matrix $M_j$ has order $n_j\times n_j$ with eigenpairs $(\alpha_{r,j},u_{r,j})$ for all $r\in I_j=\left\lbrace 1, \ldots, n_j\right\rbrace$, where the set of eigenvectors $\left\lbrace u_{r,j}:r\in I_j \right\rbrace$ is orthonormal. Let $\rho$ be a symmetric matrix of order $k\times k$ with zeros on the main diagonal and $\hat{\alpha}=(\alpha_{i_1,1},, \ldots,, \alpha_{i_k,k})$ be a $k$-tuple, where $\alpha_{i_j,j}$ is chosen from the list of eigenvalues of matrix $M_j$. Then, the matrix
\[
\rho(\mathcal{M})=
\begin{bmatrix}
M_1&\rho_{1,2}u_{i_1,1}u_{i_2,2}^t&\cdots&\rho_{1,k}u_{i_1,1}u_{i_k,k}^t\\
\rho_{1,2}u_{i_2,2}u_{i_1,1}^t&M_2&\cdots&\rho_{2,k}u_{i_2,2}u_{i_k,k}^t&\\
\vdots&\vdots&\ddots&\vdots\\
\rho_{1,k}u_{i_k,k}u_{i_1,1}^t&\rho_{2,k}u_{i_k,k}u_{i_2,2}^t&\cdots&M_k
\end{bmatrix}
\]
	
has the following set of eigenvalues
\begin{center}
$\displaystyle{\left(\bigcup_{j=1}^{k}\left\lbrace \alpha_{1,j},\, \ldots,\, \alpha_{n_j,j} \right\rbrace\setminus\left\lbrace \alpha_{i_j,j} \right\rbrace\right)\cup\left\lbrace \gamma_1,\, \ldots ,\, \gamma_k \right\rbrace},$
\end{center}
where \,$\gamma_1,\, \gamma_2,\, \ldots ,\, \gamma_k$\, are the eigenvalues of the matrix
\[\rho(\widehat{\alpha})=
\begin{bmatrix}
\alpha_{i_1,1}&\rho_{1,2}&\cdots&\rho_{1,k}\\
\rho_{1,2}&\alpha_{i_2,2}&\cdots&\rho_{2,k}\\
\vdots&\vdots&\ddots&\vdots\\
\rho_{1,k}&\rho_{2,k}&\cdots&\alpha_{i_k,k}
\end{bmatrix}.
\]
\end{lemma}

Following our convention, the above lemma reads as follows:
\begin{equation}
\sigma(\rho(\mathcal{M})) = \prod_{j=1}^k \frac{\sigma(M_j)}{\alpha_{i_j,j}} \cdot \sigma(\rho(\widetilde{\alpha})).
\end{equation}
Again, we have $g=\displaystyle{\bigvee_h a}$ in where $\pi=(\ell_1, \ell_2,..., \ell_k)$ is a segmented partition of $\ell$, $|\ell_j|=n_j$, $\mathit{a}=(g_{\ell_j})^k_{j=1}$ is the ordered assembly of internal graphs and $h$ is the external graph whose vertex set is $\pi$. Let $\mathcal{M}=\{A(g_{\ell_j})\}^k_{j=1}$ and let $\rho$ be the symmetric matrix of size $k$ defined as
\begin{equation}\label{Ro}
	\rho_{s,t}\;=\;\left\{\begin{array}{ll}
													\sqrt{n_sn_t}, & \mbox{if } \{\ell_s, \ell_t\}\in E(h)\\[0.3cm]
													            0, & \mbox{if } \{\ell_s, \ell_t\}\notin E(h)
												   \end{array}\right.
\end{equation}

If for each $j = 1, 2, \ldots, k$, the graph $g_{\ell_j}$ is $r_j$-regular, then the vector $1_{n_j\times 1}$ of size $n_j$ with all entries equal to 1 is an eigenvector of $A(g_{\ell_j})$ with eigenvalue $r_j$. Thus, if $u_{i_j, j}=\frac{1}{\sqrt{n_j}}1_{n_j\times 1}$ and $\alpha_{i_j, j}=r_j$, then $(u_{i_j, j}, \alpha_{i_j, j})$ is an eigenvector-eigenvalue pair of $A(g_{\ell_j})$. Let $\rho(\mathcal{M})$ be the corresponding matrix from Lemma (\ref{Fiedler Generalizado}). From equation (\ref{Adyacencia particionada}) in Proposition (\ref{Matriz join}), it follows that $A(g) = \rho(\mathcal{M})$. If $\rho(\hat{\alpha})$ is the matrix $k\times k$ with the vector $\hat{\alpha} = (p_1, p_2, ..., p_k)$ of internal regularities on its main diagonal, and the respective entries of the matrix $\rho$ elsewhere, then we have the following theorem published by Cardoso in \cite{Cardoso2013}.

\begin{theorem}[Cardoso's First Approach]\label{2.11}
Consider $g = \displaystyle{\bigvee_h \mathit{a}}$, where $\mathit{a} =(g_{\ell_j})^k_{j=1}$ is an ordered assembly of simple graphs such that each $g_{\ell_j}$ is $p_j$-regular, $\pi=(\ell_j)^k_{j=1}$ is a segmented partition of $\ell$, and $h$ is a simple graph over $\pi$. Let $\hat{\alpha} = (p_1, p_2, ..., p_k)$ be the vector of internal regularities and let $\rho$ be defined as in (\ref{Ro}). If $\mathcal{M} = \{A(g_{\ell_j})\}^k_{j=1}$, then $A(g)$ is equal to the corresponding matrix $\rho(\mathcal{M})$ from the Generalization of Fiedler's Lemma (\ref{Fiedler Generalizado}), and therefore,
	\begin{equation}\label{espectro adyacencia}
	\sigma(A(g))=\left(\prod_{j=1}^{k}\,\dfrac{\sigma(A(g_{\ell_j}))}{p_j}\right) \sigma(\rho(\hat{\alpha})).
\end{equation}
\end{theorem}

For the Laplacian case, consider $\mathcal{M}=\{L(g_{\ell_j})+ N_jI_{n_j}\}^k_{j=1}$. Here, $N_j$ denotes the external valence of $\ell_j$ (as defined in Proposition \ref{Matriz join}). It is well known that the matrix $L(g_{\ell_j})$ have
$(0, 1_{n_j\times 1})$ eigenpair. Thus, if $u_{i_j, j}=\frac{1}{\sqrt{n_j}}1_{n_j\times 1}$ and $\alpha_{i_j, j}=N_j$, then $(u_{i_j, j}, \alpha_{i_j, j})$ is an eigenvector-eigenvalue pair of $L(g_{\ell_j})+ N_jI_{n_j}$. Let $-\rho$ the symmetric matrix of order $k\times k$ resulting from the product of at least one to the matrix $\rho$ defined in (\ref{Ro}). Let $-\rho(\mathcal{M})$ the corresponding matrix from Lemma \ref{Fiedler Generalizado}. From the identity (\ref{Laplaciana particionada}) of Proposition \ref{Matriz join}, we get $L(g)=-\rho(\mathcal{M})$. Let $-\rho(\hat{\alpha})$ the matrix of order $k\times k$ with principal diagonal being the vector $\hat{\alpha}=(N_1, N_2,..., N_k)$ of external conections, and the corresponding entries from $-\rho$ elsewhere. Using again Lemma \ref{Fiedler Generalizado}, we obtains a proof of the following result, due to Cardoso in \cite{Cardoso2013} (Theorem 8).

\begin{theorem}[Second Cardoso's approach]\label{2.19}
Consider $g=\displaystyle{\bigvee_h a}$, where $a=(g_{\ell_j})^k_{j=1}$ is an assembly of simple graphs, $\pi=(\ell_j)^k_{j=1}$ is a segmented
partition of $\ell$ and $h$ is a simple graph over $\pi$. Let $\hat{\alpha}=(N_1, N_2,...,N_k)$ the vector of external conections, where each $N_j$ is
defined as in Proposition \ref{Matriz join}. Let $\rho$ defined as in (\ref{Ro}). If $\mathcal{M}=\{L(g_{\ell_j})+ N_jI_{n_j}\}^k_{j=1}$, then $L(g)$
is equal to the matrix $-\rho(\mathcal{M})$ given by the generalized Fiedler's \ref{Fiedler Generalizado}. In particular,
\begin{equation}\label{espectro laplaciano}
\sigma(L(g))=\left(\prod_{j=1}^{k}\,\dfrac{\sigma(L(g_{\ell_j})+ N_jI_{n_j})}{N_j}\right) \sigma(-\rho(\hat{\alpha}))
\end{equation}
\end{theorem}

\section{$\mathbb{L}$-species, and non-symmetric operads}
The goal of this section is to introduce the notion of $\mathbb{L}$-\emph{species or rigid species}, a language used to manipulates combinatorial structures constructed using \emph{labels with a total order} \cite{SpeciesBook,LibroMiguel2015}. Our main application is the study of generalized composition of graphs using the properties of $\mathbb{L}$-species.

Let $\mathbb{L}$ be the category whose objects are finite totally ordered sets, and whose morphisms are isomorphisms between them. Observe that between
any pair of liner orders of the same size there is only one isomorphism. An $\mathbb{L}$-species is a covariant
functor from $\mathbb{L}$ to the category $\mathbb{F}$ of finite sets and bijections. The $\mathbb{L}$-species with the natural
transformation as morphisms, form a category. We will say that $M = N$ if $M$ and $N$ are isomorphic $\mathbb{L}$-species. An element $m_{\ell}\in M[\ell]$ is called an $M$-structure over the linear order $\ell$, where $\ell$ is considered the set of labels of $m_{\ell}$. Two $M$-structures $m_{{\ell}_1}\in M[{\ell}_1]$ and $m_{{\ell}_2}\in M[{\ell}_2]$ are said to be isomorphic if the only isomorphism $f:{\ell}_1 \longrightarrow {\ell}_2$ is such that $M[f](m_{{\ell}_1})=m_{{\ell}_2}$.
A $\mathbb{L}$-species sending the empty set to the empty set is called \emph{positive}. A positive species sending singleton sets to singleton sets is called a \emph{delta} species. For a $\mathbb{L}-$species $M$, $M_+$ will denote the positive $\mathbb{L}$-species from $M$
\[M_+[\ell]=\left\{\begin{array}{cc}
    M[\ell] & \mbox{if}\;\;\ell\neq \emptyset,  \\
    \emptyset & \mbox{otherwise.}
\end{array}\right.
\]
More generally, for a positive integer $k$, $M_{k^+}$ denotes the $\mathbb{L}$-species
\[M_{k^+}[\ell]=\left\{\begin{array}{cc}
    M[\ell] & \mbox{if}\;\;|\ell|\geq k,  \\
    \emptyset & \mbox{otherwise.}
\end{array}\right.
\]
Given two $\mathbb{L}$-species $M$ and $N$, the operations of sum, ordinal product and ordinal substitution are defined respectively by
\begin{equation}
(M+N)[\ell]=M[\ell]\sqcup N[\ell]
\end{equation}

\begin{equation}\label{product species}
(M\lozenge N)[\ell]=\bigsqcup_{\ell_1+\ell_2=\ell} M[\ell_1]\times N[\ell_2]
\end{equation}
and
\begin{equation}\label{specie subtitution}
 M\langle N \rangle[\ell]=\bigsqcup_{\pi \in K[\ell]} N[\ell_1]\times N[\ell_2]\times \cdots \times N[\ell_{|\pi|}]\times M[\pi]
\end{equation}
The disjoint union on the definition of product runs over the decompositions of the linear order $\ell$
in two disjoint linear orders, where $\ell_1$ is an initial segment of $\ell$ and $\ell_2$ is a final segment of $\ell$. In
the definition of substitution $\pi$ runs over the set $K[\ell]$ of strong compositions or segmented partitions of $\ell$, i.e., tuples of
non-empty segments of $\ell$, $\pi=(\ell_1, \ell_2, . . . , \ell_{|\pi|})$ such that the juxtaposition $\ell_1\ell_2\cdots \ell_{|\pi|}$ is equal to $\ell$. Observe that $\pi$ is itself a totally ordered set, and hence the expression $M[\pi]$ makes sense.

The structures of $M\langle N \rangle$ are pairs of the form $(a,m)$ where $a=(n_{\ell_1} ,n_{\ell_2}, . . ., n_{\ell_{|\pi|}})=(n_{\ell})_{\ell\in \pi}$ is an ordered
assembly of $N$-structures and $m_{\pi}$ is an $M$-structure over $\pi$.

In this context the singular species is defined by
\[X[\ell]=\left\{\begin{array}{cc}
    \{\ell\} &\mbox{if}\;\;|\ell|=1,  \\
    \emptyset & \mbox{otherwise.}
\end{array}\right.
\]
is the identity for the operation of substitution, $M\langle X \rangle = X \langle M \rangle = M$, $M$ being a positive $\mathbb{L}$-species.

The empty $\mathbb{L}$-species
\[1[\ell]=\left\{\begin{array}{cc}
    \{\ell\} &\mbox{if}\;\;\ell=\emptyset,  \\
    \emptyset & \mbox{otherwise.}
\end{array}\right.
\]
is the identity with respect to the product $1\lozenge M = M \lozenge 1 = M$.

The class of positive $\mathbb{L}$-species with the operation of substitution and $X$ as identity is a monoidal
category. A monoid in this category will be called a \emph{non-symmetric operad}.
A non-symmetric operad is then a positive species plus two morphisms $\eta: M\langle M\rangle\longrightarrow M$, $e: X \longrightarrow M$, $\eta$ being
and associative ‘product’ and $e$ ‘choosing’ the identity in $M[\ell]$ for each unitary order linear $\ell$. We only consider \emph{connected} non-symmetric operads, i.e., those whose subjacent species is delta. For this case the
identity $e$ is trivially defined and we will only specify the product $\eta$.

\begin{example}\label{Join operad}  Let $\mathscr{G_+}$ be the $\mathbb{L}$-species of simple, non-empty graphs. $\mathscr{G_+}[\ell]$ is the set formed by all the simple graphs with set of vertices equal to linear order $\ell$. $\mathscr{G_+}$ is a non-symetric operad. For $((g_{\ell})_{{\ell}\in \pi}, g_{\pi})\in \mathscr{G_+}(\mathscr{G_+})[\ell]$, $\eta((g_{\ell})_{{\ell}\in \pi}, g_{\pi})$ is equal to the generalized composition of graphs $\displaystyle{\bigvee_{g_{\pi}} \mathit{a}}$
defined before, that is the graph $g$ constructed with vertices in $\ell$ as follows. Keep all the edges of the graphs in the segmented assembly, and for each edge $\{\ell_i, \ell_j\}$ of the external graph $g_{\pi}$, connect all the vertices in $\ell_i$ with all the vertices in $\ell_j$. In other words, $\{x, y\}$ is an edge in $g$ if one of the following two conditions is satisfied:
\begin{enumerate}
    \item $\{x, y\}$ is an edge in $g_{\ell}$ , for some segment $\ell$ of $\pi$.
    \item There exists an edge $\{\ell_i, \ell_j\}$ of $g_{\pi}$, such that $x\in \ell_i$ and $y\in \ell_j$.
\end{enumerate}
\end{example}

Iterate the equations (\ref{espectro adyacencia}) and (\ref{espectro laplaciano}) is not easy unless the generalized compositions of graphs can be organized into a single combinatorial and algebraic structure each time the segmented partition of internal assemblies that define them is refined. The use of operads defined by species and the operations between species allows us to describe in an organized and compact way this iteration. The species of Schröder trees is the one indicated to describe this iteration. We denote by $\mathscr{F}$ the $\mathbb{L}$-species of Schröder trees, or generalized
non-commutative parenthesizations. It satisfies the implicit equation
\begin{equation}\label{Schroder trees}
\mathscr{F} = X + E_{2^+}\langle \mathscr{F} \rangle.
\end{equation}
In this context, the implicit equation describes the structures of $\mathscr{F}[\ell]$ as plane trees whose
leaves are labelled with the linear order $\ell$ from left to right. For a tree $T\in \mathscr{F}[\ell]$, denote by
$\mathrm{Iv}(T)$ the set of the internal vertices of $T$. For an internal vertex $v$, we denote by $T_v$ the subtree of $T$ that has root $v$ and vertices all the descendants of $v$ in $T$, in otherwise, if $v$ is not an internal vertex, $T_v$ is a leaf of $T$. Let $\ell_v$ the linear suborder of $\ell$ consisting of the leaves of the subtree $T_v$. This allows us to identify the vertex $v$ with $\ell_v$, we shall use the set $\ell_v$ as a label for the vertex $v$. (See Figure \ref{f3})

Let $s(v)=\{v_1, v_2, . . . , v_k\}$ be the set of children of $v$ ordered from left to right. Each of them is either an internal vertex, or
a leaf. We denote by $\pi_v$ the segmented partition of $\ell_v$ induced by the branching at $v$: $\pi_v = (\ell_{v_1}, \ell_{v_2},..., \ell_{v_k})$

\begin{figure}[h]
    \centering
    \includegraphics[width=0.7\textwidth]{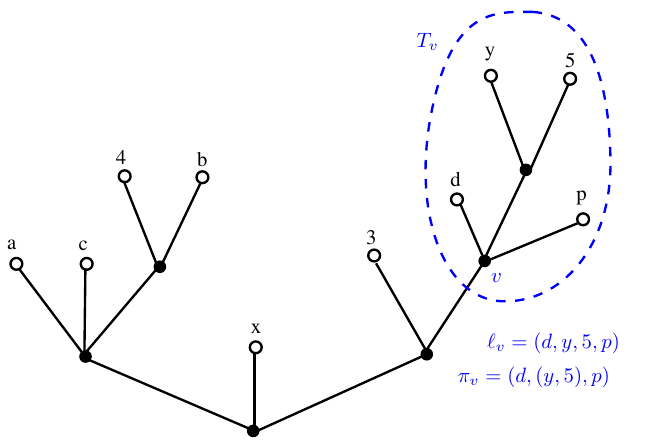}
    \caption{Schröder tree over linear order $\ell=(a,c,4,b,x,3,d,y,5,p)$}
    \label{f3}
\end{figure}
For a $\mathbb{L}$-species $M$ having the form
\[M = X + M_{2^+},\]
the $\mathbb{L}$-species of $M$-enriched Schröder trees is the solution to the implicit equation
\begin{equation}\label{enriched schroder tree}
\mathscr{F}_M = X + M_{2^+}\langle \mathscr{F}_M\rangle
\end{equation}

Using (\ref{enriched schroder tree}), we obtain the following recursive description $\mathscr{F}_M[\ell]$. If $\ell$ is an unitary linear order, the only
tree in $\mathscr{F}_M[\ell]$ consists of only one leaf labelled with the element of $\ell$ (the singleton-leaf tree).
\begin{figure}[h]
    \centering
    \includegraphics[width=1\textwidth]{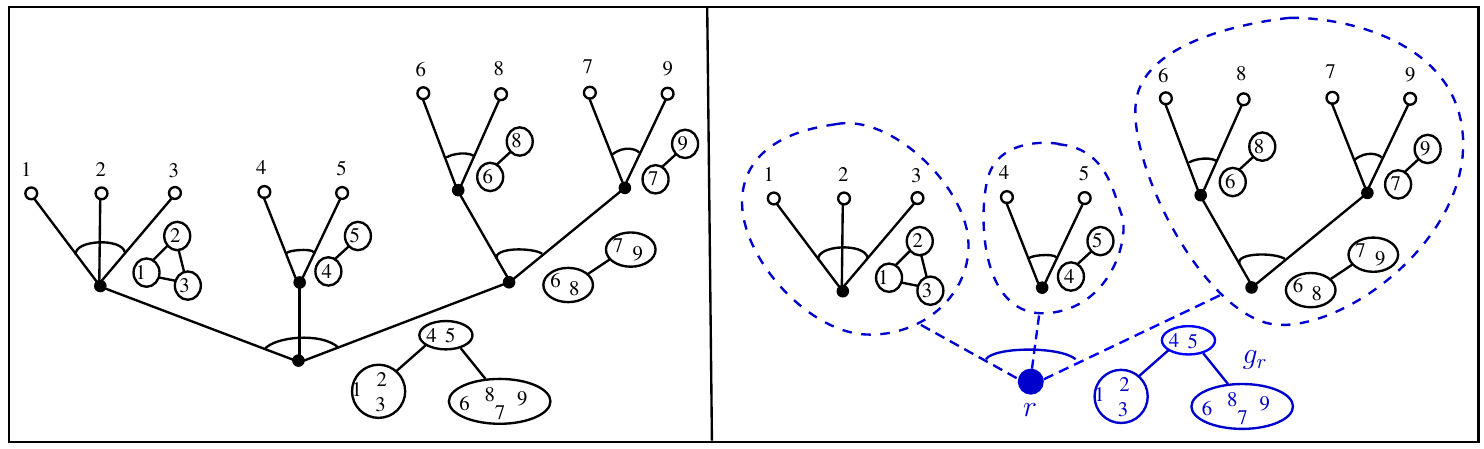}
    \caption{Schröder tree enriched with simple graph}
    \label{f4}
\end{figure}
If $|\ell|\geq 2$, an element $\mathscr{T}$ in $\mathscr{F}_M[\ell]$ is a pair $((\mathscr{T}_{v_1}, \mathscr{T}_{v_2},..., \mathscr{T}_{v_j}) , m_r)$ where $r$ is the root of $\mathscr{T}$, $v_1, v_2, . . . , v_j$ are the children of $r$ ordered from left to right, $m_r\in M_{2^+}[\pi_r]$ is the structure attached to the root of $\mathscr{T}$ and $(\mathscr{T}_{v_1}, \mathscr{T}_{v_2},..., \mathscr{T}_{v_j})$ is the ordered assembly of smaller trees whose roots are the children of $r$, that is to say, $\mathscr{T}_{v_i}\in \mathscr{F}_M[\ell_{v_i}]$ for $i=1,2,..., j$. Iterating this recursive description we get an explicit expression of the set $\mathscr{F}_M[\ell]$;
\[\mathscr{F}_M[\ell] = X[\ell] \cup \bigcup_{T\in \mathscr{F}[\ell]}  \{(T, (m_v)_{v\in \mathrm{Iv}(T)}): m_v\in M_{2^+}[\pi_v]\} \]
In other words, a tree $\mathscr{T}$ in $\mathscr{F}_M[\ell]$ is a Schröder tree $T \in \mathscr{F}[\ell]$ together with a structure $m_v\in
M_{2^+}[\pi_v]$ for each internal vertex $v\in \mathrm{Iv}(T)= \mathrm{Iv}(\mathscr{T})$ (see Figure \ref{f4}).

If $(M,\eta)$ is a non-symmetric operad, for a tree $\mathscr{T} \in \mathscr{F}_M[\ell]$, $\widehat{\eta}(\mathscr{T})$ is the element of $M[\ell]$ obtained by applying the operad product recursively on each level of the tree (see Figure \ref{f5}).
\begin{figure}[h]
    \centering
    \includegraphics[width=1\textwidth]{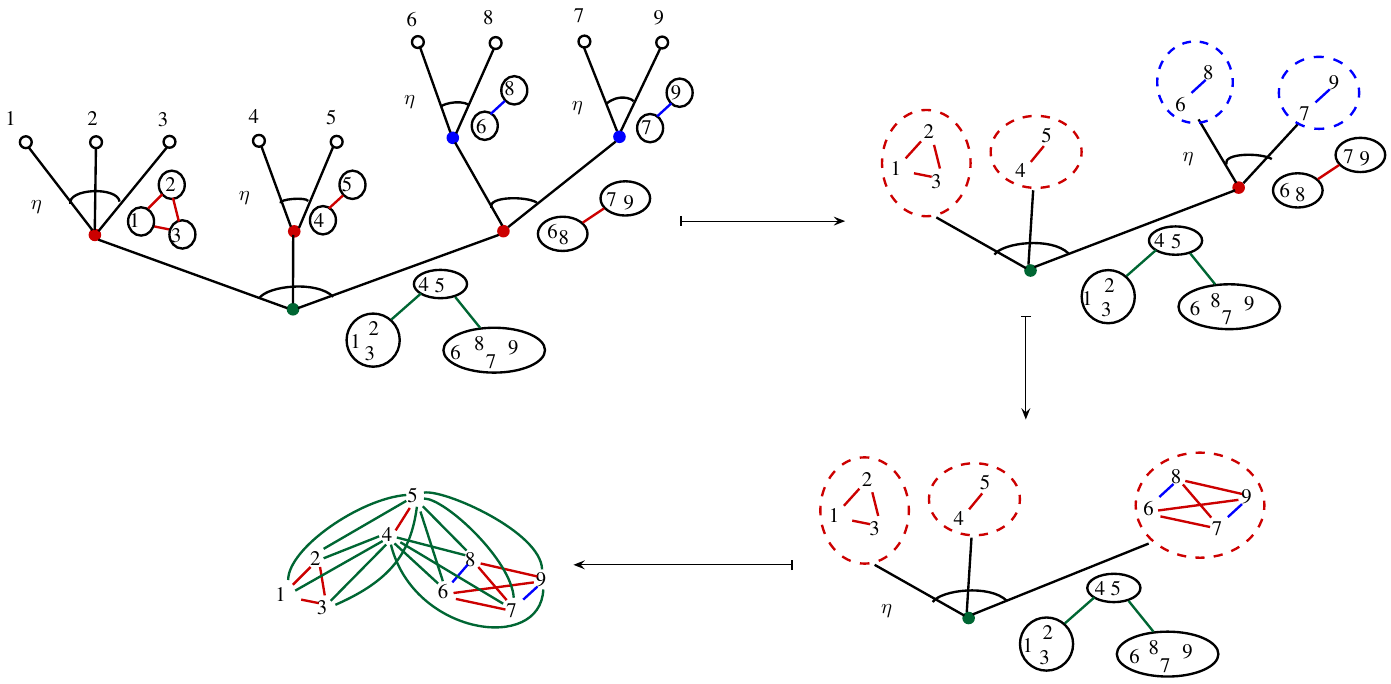}
    \caption{Iterative product $\widehat{\eta}$ in $\mathscr{F}_{\mathscr{G}_+}$}
    \label{f5}
\end{figure}
In this way, if $\mathscr{T}$ is the singleton-leaf tree, $\widehat{\eta}(\mathscr{T})$ is equal to the singleton structure in $M[\ell]$. Otherwise, we define $\widehat{\eta}(\mathscr{T})$ recursively by
\begin{equation}\label{operad product schroder tree}
\widehat{\eta}(\mathscr{T})=\eta((\widehat{\eta}(\mathscr{T}_{v_1}), \widehat{\eta}(\mathscr{T}_{v_2}),..., \widehat{\eta}(\mathscr{T}_{v_j})) , m_r)
\end{equation}

The previous recursive procedure can be replaced by any other that systematically apply the products on each internal vertex of the tree, finishing with the root. By associativity of the operad product the result will be the same. $\widehat{\eta}:\mathscr{F}_M \longrightarrow M$ is a species morphism and has been studied in \cite{Miguel-Jean2014} to give formulas of antipodes of Hopf algebras that are constructed with combinatorial objects.

\section{Iteration of Cardoso's formulas}
In this section we will get use of the operad  $(\mathscr{G}_+, \eta)$ in where $\eta$ is the generalized composition of graphs (see example \ref{Join operad}) to study the iterations of the equations (\ref{espectro adyacencia}) and (\ref{espectro laplaciano}), using the product $\widehat{\eta}$ on Schröder trees $\mathscr{F}_{\mathscr{G}_+}$. Given a tree $\mathscr{T} \in \mathscr{F}_{\mathscr{G}_+}[\ell]$ we denotes the product $\widehat{\eta}(\mathscr{T})$ simply by $\bigvee \mathscr{T}$. Given a vertex $v$ of $\mathscr{T}$, we denote by $\mathscr{T}_{v}$ the subtree of $\mathscr{T}$ whose set of leaves is $\ell_v$, and $g_{\ell_v}$ will denote the product $\bigvee \mathscr{T}_v$ whenever $g=\bigvee \mathscr{T}$. If $v\in \mathrm{Iv}(\mathscr{T})$ and $s(v)=\{v_1,v_2,...,v_k\}$ the set of childrens of $v$, we denotes by  $a_v$ the ordered assembly of graphs $\left(\bigvee \mathscr{T}_{v_i}\right)^{k}_{i=1}$. That is, the components of $a_v$ are obtained as the result of applying the product $\bigvee$ to the subtrees which have as root any children of $v$.

\subsection{Adjacency spectrum iteration}
For the case in which the components of $a_v$ are regular graphs, we denote by $\mathrm{reg}(a_v)$ the product of the regularities of the components in $a_v$, this is $\prod^k_{i=1}\mathrm{reg}\left(\bigvee \mathscr{T}_{v_i}\right)$. If $v\in \mathrm{Iv}(\mathscr{T})$ we denote by $A(a_v, g_v)$ to the matrix  $\rho(\widehat{\alpha})$ from Theorem \ref{2.11} after applying the product $\displaystyle{\bigvee_{g_v} a_v}$ with $g_v\in \mathscr{G}_+[\pi_v]$ . From equation (\ref{espectro adyacencia}) we get:
\begin{equation}\label{espectro adyacencia Join arbol}
\sigma\left(A\left(\bigvee_{g_v} a_v\right)\right)=\dfrac{\sigma(A(a_v))}{\mathrm{reg}(a_v)}\cdot \sigma(A(a_v, g_v)).
\end{equation}

Iterating this identity over each internal vertex $v$ of $\mathscr{T}$, leads to the following result.

\begin{theorem}\label{adjacency spectrum iterated}
Let $\mathscr{T}\in \mathscr{F}_{\mathscr{G}_+}[\ell]$ be a factorization of a graph $g$, that is, $\displaystyle{\bigvee \mathscr{T}=g}$. Suppose that for each internal vertex $v$ of $\mathscr{T}$, $a_v$ is an ordered assembly of regular graphs. Then
\begin{equation}\label{espectro adyacencia sobre un arbol enriquecido}
\sigma(A(g))=0^{|\ell|}\cdot \prod_{v\in \mathrm{Iv}(\mathscr{T})}\dfrac{\sigma(A(a_v,g_v))}{\mathrm{reg}(a_v)}
\end{equation}
\end{theorem}

\begin{proof}
Let $\psi$ be the function given by
\[A\left(\bigvee \mathscr{T}\right) \mapsto \displaystyle{0^{|\ell|}\cdot \prod_{v\in \mathrm{Iv}(\mathscr{T})}\dfrac{\sigma(A(a_v,g_v))}{\mathrm{reg}(a_v)}}.\]

Let $r$ be the root of $\mathscr{T}$, let $s(r)=\{r_1, r_2,...,r_k\}$ be the set of childrens of $r$. Then, for each $i=1,2,3,...,k$, we have
\begin{equation*}
\dfrac{\psi\left(A\left(\bigvee \mathscr{T}_{r_i}\right)\right)}{\mathrm{reg}\left(\bigvee \mathscr{T}_{r_i}\right)}=\left(0^{|\ell_{r_i}|}\cdot \prod_{w\in \mathrm{Iv}(\mathscr{T}_{r_i})}\dfrac{\sigma(A(a_w,g_w))}{\mathrm{reg}(a_w)}\right)\cdot \dfrac{1}{\mathrm{reg}\left(\bigvee \mathscr{T}_{r_i}\right)}.
\end{equation*}
Therefore,
\begin{eqnarray*}
\prod^k_{i=1} \dfrac{\psi\left(A\left(\bigvee \mathscr{T}_{r_i}\right)\right)}{\mathrm{reg}\left(\bigvee \mathscr{T}_{r_i}\right)}&=&0^{|\ell|}\cdot \left(\prod_{w\in \mathrm{Iv}(\mathscr{T}) \atop w\neq r} \dfrac{\sigma(A(a_w,g_w))}{\mathrm{reg}(a_w)}\right)\cdot \dfrac{1}{\mathrm{reg}(a_r)}
\end{eqnarray*}

which implies that
\begin{eqnarray*}
\dfrac{\psi(A(a_r))}{\mathrm{reg}(a_r)}\cdot \sigma(A(a_r, g_r))&=&0^{|\ell|}\left(\prod_{w\in \mathrm{Iv}(\mathscr{T})\atop w\neq r} \dfrac{\sigma(A(a_w,g_w))}{\mathrm{reg}(a_w)}\right)\cdot \dfrac{\sigma(A(a_r,g_r))}{\mathrm{reg}(a_r)}\\
&=&\psi\left(A\left(\bigvee \mathscr{T}\right)\right).
\end{eqnarray*}
In conclusion, $\psi$ satisfies the recursive equation (\ref{espectro adyacencia Join arbol}), so $\psi=\sigma$.
\end{proof}

If we denotes by $a^+_v$ the ordered assembly of graphs whose components are the graphs in $a_v$ with at least two vertices, then the equation (\ref{espectro adyacencia sobre un arbol enriquecido}) may be simplified as follows:
\begin{equation}\label{espectro ady join arbol enriq simplificada}
\sigma(A(g))=\prod_{v\in \mathrm{Iv}(\mathscr{T})}\dfrac{\sigma(A(a_v, g_v))}{\mathrm{reg}(a^+_v)}.
\end{equation}
The quotient $\displaystyle{\dfrac{\sigma(A(a_v, g_v))}{\mathrm{reg}(a^+_v)}}$ on the right-hand side below is equal to $\sigma(A(a_v, g_v))=\sigma(A(g_v))$ only when $a_v$ is an assembly of singular graphs, that is, the childrens of $v$ are leaves.

\begin{example}
Let $\bigvee \mathscr{T}=g$ be the graph of Figure \ref{f5} and let $\textcolor[rgb]{0.00,0.50,0.25}{v_1}\textcolor[rgb]{0.98,0.00,0.00}{v_2v_3v_4}\textcolor[rgb]{0.00,0.00,1.00}{v_5v_6}$ be the total order on $\mathrm{Iv}(\mathscr{T})$ obtained after applying the depth-first search algorithm to $\mathscr{T}$. From equation (\ref{espectro ady join arbol enriq simplificada}) we have
\begin{eqnarray*}
    \sigma(A(g))&=&\sigma(A(a_{\textcolor[rgb]{0.00,0.50,0.25}{v_1}}, g_{\textcolor[rgb]{0.00,0.50,0.25}{v_1}}))\cdot \sigma(A(g_{\textcolor[rgb]{0.98,0.00,0.00}{v_2}})) \cdot \sigma(A(g_{\textcolor[rgb]{0.98,0.00,0.00}{v_3}}))\cdot \sigma(A(a_{\textcolor[rgb]{0.98,0.00,0.00}{v_4}}, g_{\textcolor[rgb]{0.98,0.00,0.00}{v_4}}))\cdot \sigma(A(g_{\textcolor[rgb]{0.00,0.00,1.00}{v_5}}))\cdot \sigma(A(g_{\textcolor[rgb]{0.00,0.00,1.00}{v_6}}))\\
    &=&\dfrac{\sigma\left(\begin{array}{ccc} 2 & \sqrt{6} & 0\\ \sqrt{6}&1&\sqrt{8}\\0&\sqrt{8}&3\end{array}\right)}{2\cdot 1\cdot 3}\cdot \sigma\left(\begin{array}{ccc} 0 & 1 & 1\\ 1& 0 & 1\\ 1&1&0\end{array}\right)\cdot \left(\sigma \left(\begin{array}{cc}
       0  & 1 \\
       1  & 0
    \end{array}\right)\right)^3 \cdot \dfrac{\sigma\left(\begin{array}{cc}1&2\\2&1\end{array}\right)}{1^2}\\
    &=&\dfrac{\lambda_1\cdot\lambda_2\cdot\lambda_3}{2\cdot 1\cdot 3}\cdot [(-1)^2\cdot 2]\cdot [(-1)^3\cdot (1)^3]\cdot \dfrac{1\cdot 3}{1^2}\cdot \\
&=&(-1)^5\cdot 1\cdot \lambda_1\cdot\lambda_2\cdot\lambda_3,
\end{eqnarray*}
where $\lambda_1\approx 5.65$, $\lambda_2\approx -2.05$ y  $\lambda_3\approx 2.40$.
\end{example}

 \subsection{Laplacian spectrum iteration}
Before initiating the study concerning the iteration of the Laplacian spectrum we will need the following definition.
\begin{definition}\label{desplazamiento monomial}
Let $\omega=0^{i_0}\lambda^{i_1}_1\cdots \lambda^{i_k}_{k}$ be a monomial with real variables. Let $s$ be a real number. The map $\omega \mapsto \varphi_s(\omega)$ is defined as
\begin{enumerate}
\item $(\lambda_1+s)^{i_1}(\lambda_2+s)^{i_2}\cdots (\lambda_k+s)^{i_k}$, if $i_0=0, 1$
\item $s^{i_0-1}(\lambda_1+s)^{i_2}\cdots (\lambda_k+s)^{i_k}$, if $i_0\geq 2$.
\end{enumerate}
\end{definition}
The operator $\varphi_s$ downgrade the multiplicity of $0$ by one (when it applies) and then add the value $s$ to each variable of $\omega$. It is straighforward to verify that $\varphi_{s_1}\circ \varphi_{s_2}=\varphi_{s_1+s_2}$, as soon $s_1, s_2$ are positive real numbers. However, we don't have equality if $s_2=0$ and $i_0\geq 2$.

Let  $\omega_1$ and $\omega_2$ be monomials of real numbers. If the multiplicity of the letter $0$ is null on $\omega_1$ or $\omega_2$, then
$\varphi_s(\omega_1\omega_2)=\varphi_s(\omega_1)\varphi_s(\omega_2)$. In general, the map $\varphi_s$ is not multiplicative.

Suppose as before that $\pi=(\ell_j)^k_{j=1}$ is a partition of $\ell$, $a=(g_{\ell_j})^k_{j=1}$ is an assembly of simple graphs and $h$ is a graph with vertex set $\pi$. From definition \ref{desplazamiento monomial} it follows that
\begin{equation}
\dfrac{\sigma(L(g_{\ell_j})+N_jI_{n_j})}{N_j}=\varphi_{N_j}(\sigma(L(g_{\ell_j}))).
\end{equation}
We will denotes convenietly the matrix $-\rho(\hat{\alpha})$ from theorem \ref{2.19} by $\mathscr{L}(h)$. In this way, the identity (\ref{espectro laplaciano}) can be stated as
\begin{equation}\label{varphi laplaciano}
\sigma\left( L\left(\bigvee_{h} a \right)\right)=\left(\prod^k_{j=1} \varphi_{N_j}(\sigma(L(g_{\ell_j})))\right)\cdot \sigma(\mathscr{L}(h))
\end{equation}
Notice that if $|\ell_j|=1$, then  $\varphi_{N_j}(\sigma(L(g_{\ell_j}))$ is the empty monomial. Therefore,
\begin{equation}\label{espectro Laplace simplificada}
\sigma\left( L\left(\bigvee_{h} a\right) \right)=\left(\prod_{1\leq j \leq k \atop |\ell_j|>1 } \varphi_{N_j}(\sigma(L(g_{\ell_j}))\right)\cdot \sigma(\mathscr{L}(h))
\end{equation}
Let $\mathscr{G}_c$ be the $\mathbb{L}$-species of connected simple graphs and $\mathscr{T}\in \mathscr{F}_{\mathscr{G}_c}$ a Schröder tree. Let $w,u\in \mathrm{Iv}(\mathscr{T})$ such that $u$ is a parent of $w$. Consider
\[N_w:=\sum_{\{\ell_x,\ell_w\} \in E(g_u)} |\ell_x|\]
The value $N_{w}$ is the external valency of the segment $\ell_w\in \pi_{u}$, according to the graph $g_u$. Let $r$ the root of $\mathscr{T}$,
$g=\bigvee \mathscr{T}$ and $w_1, w_2,..., w_k$ the internal vertices of $\mathscr{T}$ that are children $r$, from equation (\ref{espectro Laplace simplificada}) it follows that

\begin{equation}\label{espectro Laplace nivel 1}
\sigma (L(g))=\left(\prod^k_{j=1} \varphi_{N_{w_j}}\left(\sigma\left(L\left(\bigvee \mathscr{T}_{w_j}\right)\right)\right)\right)\cdot \sigma(\mathscr{L}(g_r)).
\end{equation}

In the other hand, if $r w_1\cdots w_k w$ is the path on $\mathscr{T}$ from the root of $r$ in $\mathscr{T}$
to the internal vertex $w$ on $\mathscr{T}$, we denotes by $N_{rw}$ the sum of all the external valencies $N_{w_i}$:
\[N_{rw}=N_{w_1}+N_{w_2}+\cdots + N_{w_k} + N_w.\]

It is not hard to verify that $N_{rw}$ is the external valency of the graph $g_{\ell_w}$, that is, $N_{rw}$ is the number of vertices of the graph $g$ which are adyacents to all elements of $\ell_w$:

\[N_{vw}=|\{x\in \ell : \{x, u\}\in E(g)\,\, \forall u\in \ell_w\}|\]

The iteration of equation (\ref{espectro Laplace nivel 1}) can be cumbersome in general, however, since tree $\mathscr{T}$ is enriched with connected graphs, the iteration of equation (\ref{espectro Laplace nivel 1}) is easy to perform from the properties of the function $\varphi_s$. In this way, we obtain the theorem

\begin{theorem}\label{Laplacian spectrum iterated}
Let $\ell$ be a non-singular linear order and $\mathscr{T}\in \mathscr{F}_{\mathscr{G}_c}[\ell]$ a factorization of the graph $g$, that is, $\displaystyle{\bigvee \mathscr{T}=g}$. If $r$ is the root of $\mathscr{T}$, then
\begin{equation}\label{lapaciano iterado}\sigma(L(g))=\left(\prod_{w\in \mathrm{Iv}(\mathscr{T})\atop w\neq r} \varphi_{N_{rw}}(\sigma(\mathscr{L}(g_w)))\right)\cdot \sigma(\mathscr{L}(g_r))\end{equation}
\end{theorem}

\begin{proof}
Let $\psi$ the application sending $L(g)$ to \[\left(\prod_{w\in \mathrm{Iv}(\mathscr{T})\atop w\neq r} \varphi_{N_{rw}}(\sigma(\mathscr{L}(g_w))\right)\cdot \sigma(\mathscr{L}(g_r)),\]
If $w_1, w_2,..., w_k$ are the internal vertices of $\mathscr{T}$ which are children of $r$, then for each $j=1,2,...,k$ we have
\[\varphi_{N_{w_j}}\left(\psi \left(L\left(\bigvee \mathscr{T}_{w_j}\right)\right)\right)=\varphi_{N_{w_j}}\left(\left(\prod_{x\in \mathrm{Iv}(\mathscr{T}_{w_j})\atop x\neq w_j} \varphi_{N_{w_jx}}(\sigma(\mathscr{L}(g_x)))\right)\cdot \sigma(\mathscr{L}(g_{w_j}))\right).\]
As $N_w>0$ for all internal vertex of $\mathscr{T}$ such that $w\neq r$, then $N_{w_jx}>0$ for all $x\in (\mathrm{Iv}(\mathscr{T}_{w_j})-\{w_j\})$. Therefore, the multiplicity of the variable $0$ in $\varphi_{N_{w_jx}}(\sigma(\mathscr{L}(g_x))$ is zero for all $x\in (\mathrm{Iv}(\mathscr{T}_{w_j})-\{w_j\})$. From the properties of the function $\varphi_s$ it follows that

\begin{eqnarray*}\varphi_{N_{w_j}}\left(\psi \left(L\left(\bigvee \mathscr{T}_{w_j}\right)\right)\right)&=&\left(\prod_{x\in \mathrm{Iv}(\mathscr{T}_{w_j})\atop x\neq w_j} \varphi_{N_{w_j}+N_{w_jx}}(\sigma(\mathscr{L}(g_x)))\right)\cdot \varphi_{N_{w_j}}(\sigma(\mathscr{L}(g_{w_j})))\\
&=&\prod_{x\in \mathrm{Iv}(\mathscr{T}_{w_j})} \varphi_{N_{rx}}(\sigma(\mathscr{L}(g_x))).
\end{eqnarray*}

From here, we have
\begin{eqnarray*}
\left(\prod^k_{j=1}\varphi_{N_{w_j}}\left(\psi \left(L\left(\bigvee \mathscr{T}_{w_j}\right)\right)\right) \right)\cdot \sigma(\mathscr{L}(g_r))&=& \left(\prod_{x\in \mathrm{Iv}(\mathscr{T})\atop x\neq r} \varphi_{N_{rx}}(\sigma(\mathscr{L}(g_x)))\right)\cdot \sigma(\mathscr{L}(g_r))\\
&=&\psi(L(g))
\end{eqnarray*}
Since $\psi$ satisfy the recursive equation $(\ref{espectro Laplace nivel 1})$, then we concludes that $\psi=\sigma$.
\end{proof}

\begin{example}
Let $\bigvee \mathscr{T}=g$ be the graph of Figure \ref{f5} and let $\textcolor[rgb]{0.00,0.50,0.25}{v_1}\textcolor[rgb]{0.98,0.00,0.00}{v_2v_3v_4}\textcolor[rgb]{0.00,0.00,1.00}{v_5v_6}$ be the total order on $\mathrm{Iv}(\mathscr{T})$ obtained after applying the depth-first search algorithm to $\mathscr{T}$. From equation (\ref{lapaciano iterado}) it follows that
\begin{eqnarray*}
\sigma(L(g))&=&\sigma(\mathscr{L}(g_{\textcolor[rgb]{0.00,0.59,0.00}{v_1}}))\cdot \varphi_{N_{\textcolor[rgb]{0.00,0.59,0.00}{v_1}\textcolor[rgb]{0.98,0.00,0.00}{v_2}}}(\sigma(\mathscr{L}(g_{\textcolor[rgb]{0.98,0.00,0.00}{v_2}})))\cdot \varphi_{N_{\textcolor[rgb]{0.00,0.59,0.00}{v_1}\textcolor[rgb]{0.98,0.00,0.00}{v_3}}}(\sigma(\mathscr{L}(g_{\textcolor[rgb]{0.98,0.00,0.00}{v_3}})))\cdot \varphi_{N_{\textcolor[rgb]{0.00,0.59,0.00}{v_1}\textcolor[rgb]{0.98,0.00,0.00}{v_4}}}(\sigma(\mathscr{L}(g_{\textcolor[rgb]{0.98,0.00,0.00}{v_4}})))\\
&& \cdot \varphi_{N_{\textcolor[rgb]{0.00,0.59,0.00}{v_1}\textcolor[rgb]{0.00,0.00,1.00}{v_5}}}(\sigma(\mathscr{L}(g_{\textcolor[rgb]{0.00,0.00,1.00}{v_5}}))) \cdot\varphi_{N_{\textcolor[rgb]{0.00,0.59,0.00}{v_1}\textcolor[rgb]{0.00,0.00,1.00}{v_6}}}(\sigma(\mathscr{L}(g_{\textcolor[rgb]{0.00,0.00,1.00}{v_6}})))\\
&=&\sigma\left(\begin{array}{rrr}
2&-\sqrt{6} & 0 \\
-\sqrt{6}&7&-\sqrt{8}\\
0&-\sqrt{8}&2
\end{array}\right)\cdot \varphi_{\textcolor[rgb]{0.00,0.50,0.00}{2}}\left(\sigma\left(\begin{array}{rrr}
2&-1&-1 \\
-1&2&-1\\
-1&-1&2
\end{array}\right)\right)\cdot \varphi_{\textcolor[rgb]{0.00,0.50,0.00}{7}}\left(\sigma\left(\begin{array}{rr}
   1  &  -1\\
   -1  &  1\\
\end{array}\right)\right)\\
& &\cdot \varphi_{\textcolor[rgb]{0.00,0.50,0.00}{2}}\left(\sigma\left(\ \begin{array}{rr}
    2 & -2 \\
    -2 & 2
\end{array}\right)\right) \cdot \left(\varphi_{\textcolor[rgb]{1.00,0.00,0.00}{2}+\textcolor[rgb]{0.00,0.50,0.00}{2}}\left(\sigma\left(\ \begin{array}{rr}
    1 & -1 \\
    -1 & 1
\end{array}\right)\right)\right)^2\\
&=&(0\cdot 2\cdot 9)\cdot \varphi_2(0\cdot 3^2)\cdot \varphi_7(2\cdot 0)\cdot\varphi_2(0\cdot 4)\cdot (\varphi_4(2\cdot 0))^2\\
&=& 0\cdot 2\cdot 9 \cdot 5^2\cdot 9\cdot 6 \cdot 6^2 \\
&=& 0\cdot 2 \cdot 5^2 \cdot 6^3  \cdot 9^2
\end{eqnarray*}
\end{example}

\subsection{Colored iteration of Cardoso's formulas}
Let $E_+$ be the positive exponential $\mathbb{L}$-species ($E_+[\ell]=\{\ell\}$ with $|\ell|\geq 1$). Since every simple graph is an assembly of connected simple graphs, from the definition of substitution we get $E_+(\mathscr{G}_c)=\mathscr{G}_+$. In \cite{Miguel-Jean2014} the generalized composition of graphs is used to define the notion of \emph{divisibility} between assemblies of connected simple graphs. More precisely, we have the following definition.

\begin{definition}\label{divisibilidad}
Let $a_1, a_2$ be two ordered assemblies of connected simple graphs. We say that $a_1$ \emph{divides} $a_2$ if there exists a simple graph $h$ (not necessarily connected) with vertices on the subyacent partition of $a_1$ such that $\displaystyle{\bigvee_{h} a_1}=a_2$. This is to say, $a_1$ is a refinement of $a_2$ from the partitioned composition of graphs. In this case, $h$ is the quotient graph denoted by $a_2/a_1$.
\end{definition}

The colours on the edges of the graph $g$ on Figure \ref{f5} is not casual; this colouring is in natural correspondence with the factorization $\bigvee \mathscr{T}=g$. The idea is to color the internal vertices of a Schröder tree $\mathscr{T}\in \mathscr{F}_{\mathscr{G}_c}$ with the successor of the depth that these have in the tree, starting from the root which has color 1. Since each internal vertex is the instruction of the product $\bigvee$, the color of the internal vertex is given to the edges of the connected graph enriches it. After executing all the products indicated by each internal vertex, we obtain a graph whose edges are colored. This type of coloring on the graph has been studied in \cite{Miguel-Jean2014, LibroMiguel2015} and is called \emph{admissible colouration}.

\begin{figure}
    \centering
    \includegraphics[width=0.5\linewidth]{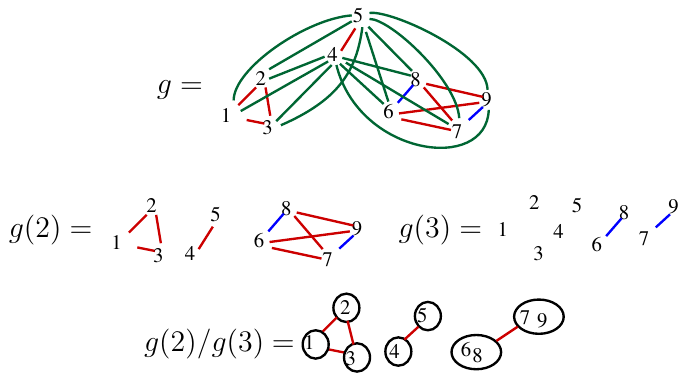}
    \caption{Subgraphs induced by admissible colouring}
    \label{f6}
\end{figure}

\begin{definition}\label{coloracion admisible grafos conexos} Let $c$ be a colouration on the edges of a connected graph $g$. That is, $c$ is a function from the set of edges of $g$ to the set of positive integers. For every $i>1$ on the image of $c$, we denotes by $g(i)$ the subgraph of $g$ obtained by erasing all edges of colour $j$, $1\leq j<i$. For $i=1$, let $g(1):=g$. The colouring $c$ is say to be admissible if:
\begin{enumerate}
\item for each $i\geq 2$ on the image of $c$, $g(i)$ divides $g(i-1)$.
\item if an incident edge of a vertex $v$ of $g$ has colour $i\geq 2$, then there exists at least one edge of colour $i-1$, also incident to $v$.
\end{enumerate}
\end{definition}

The Figure \ref{f6} illustrates the notion of subgraph induced by an admissible colouring. We denote by $a_c(i)$ the assembly formed by the connected graphs in $g(i)/g(i+1)$ that are not singular. For each $h\in a_c(i)$, let $N_h$ be the number of edges of $g$ with color $j<i$ that are adjacent to $\displaystyle{\bigcup_{\ell\in V(h)} \ell}$, and $g_h(i+1)$ the subassembly of $g(i+1)$ with partition underlying $V(h)$. If $r_i$ is the regularity of the graph $g_h=\displaystyle{\bigvee_{h}g_h(i+1)}$, from the equations (\ref{espectro ady join arbol enriq simplificada}), (\ref{lapaciano iterado}), we obtains the following result.

\begin{theorem}\label{colored spectrum}
Let $g$ be a connected simple graph and let $c:E(g)\longrightarrow \{1, 2, 3,..., k\}$ an admissible colouration over the edges of $g$.
\begin{enumerate}
\item If $g(i)$ is an assembly of regular connected graphs, for every $2\leq i\leq k$. Then,
   \begin{equation}\label{colored adjacent spectrum}
       \sigma(A(g))=\prod^{k}_{i=1}\prod_{h\in a_c(i)}\dfrac{\sigma(A(g_h(i+1), h))}{\mathrm{reg}(g^+_h(i+1))}
   \end{equation}

\item \begin{equation}\label{laplaciano coloreado}
\sigma(L(g))=\prod^k_{i=2}\left(\prod_{h\in a_c(i)} \varphi_{N_h}\left(\sigma\left(\mathscr{L}(h)\right)\right)\right)\cdot \sigma\left(\mathscr{L}\left(g(1)/g(2)\right)\right).
\end{equation}

\end{enumerate}

\end{theorem}

\subsection{The complement operation and the generalized composition of graphs}
 It is well known that if $g$ is a regular graph with $n$ vertices and the eigenvalues of $A(g)$ are $\lambda_1, \lambda_2,...,\lambda_{n-1}, \lambda_n=\mathrm{reg}(g)$, then the complement of the graph $g$ denoted by $g^c$ is an $(n-\mathrm{reg}(g)-1)$-regular graph and the eigenvalues of $A(g^c)$ are $-1-\lambda_1, -1-\lambda_2,...,-1-\lambda_{n-1}, n-\mathrm{reg}(g)-1$. On the other hand, the complement operation is an automorphism of the operad $(\mathscr{F}_{\mathscr{G}_+}, \bigvee)$, that is, if $\mathscr{T}=(T, (g_v)_{v\in \mathrm{Iv}(T)})$ belongs to $\mathscr{F}_{\mathscr{G_+}}[\ell]$ where $T\in\mathscr{F}[\ell]$, $n=|\ell|$ and $g_v\in \mathscr{G}_+[\pi_v]$, then the complement of $\mathscr{T}$ is the pair $(T, (g^c_v)_{v\in \mathrm{Iv}(T)})$, denoted by $\mathscr{T}^c$. If $g=\bigvee \mathscr{T}$, then
\[g^c=\left(\bigvee \mathscr{T}\right)^c=\bigvee \mathscr{T}^c.\]

 Suppose that $g_{\ell_v}$ is a regular graph for every $v\in \mathrm{Iv}(\mathscr{T})$, from generalized Fiedler's lemma (Cardoso et al., \cite{Cardoso2013}) one can verify that $\mathrm{reg}(g_{\ell_v})$ is an eigenvalue of $A(a_v, g_v)$ and from equation (\ref{espectro adyacencia Join arbol}) it follows that
\begin{equation}
\sigma(A(g^c_{\ell_v}))=\left(-1-\left(\dfrac{\sigma(A(a_v))}{\mathrm{reg}(a_v)}\right)\right)\cdot \left(-1-\left(\dfrac{\sigma(A(a_v,g_v))}{\mathrm{reg}(g_{\ell_v})}\right)\right)\cdot(|\ell_v|-\mathrm{reg}(g_{\ell_v})-1)
\end{equation}

where $-\sigma(B)$ denotes changing the signs of the factors of $\sigma(B)$. From this equation we deduce
\begin{equation}\label{ady externa complement}
\sigma(A(a_v^c, g_v^c))=-1-\left(\dfrac{\sigma(A(a_v,g_v))}{\mathrm{reg}(g_{\ell_v})}\right)\cdot(|\ell_v|-\mathrm{reg}(g_{\ell_v})-1).
\end{equation}

If $r$ is the root of $\mathscr{T}$, from equations (\ref{espectro ady join arbol enriq simplificada}) and (\ref{ady externa complement}) we deduce
\begin{eqnarray*}
  \sigma(A(g^c))&=&\prod_{v\in \mathrm{Iv}(\mathscr{T})\atop v\neq r}\left(-1-\left(\dfrac{\sigma(A(a_v,g_v))}{\mathrm{reg}(a^+_v)}\right)\right)\\
  &&\cdot\left(-1 -\left( \dfrac{\sigma(A(a_r,g_r))}{\mathrm{reg}(g)} \right)\right)
\cdot (n-\mathrm{reg}(g)-1)\\
\end{eqnarray*}
 which is rewritten as
\begin{equation}
  \sigma(A(g^c))=\prod_{v\in \mathrm{Iv}(\mathscr{T})}\left(-1-\left(\dfrac{\sigma(A(a_v,g_v))}{\mathrm{reg}(g_{\ell_v})}\right)\right)\\
\cdot (n-\mathrm{reg}(g)-1)
\end{equation}

The following corollaries were presented in \cite{Cardoso2013} as an immediate consequence of theorem \ref{2.19}, these results will allow us to interpret the iterated Laplacian spectrum of the complement of a graph.

\begin{corollary}\label{nulidad laplaciana}
Let $g=\displaystyle{\bigvee_h a}$, where $a=(g_{\ell_j})^k_{j=1}$ is an ordered assembly of simple graphs, $\pi=(\ell_j)^k_{j=1}$ is a segmented partition of the linear order $\ell$, and $h$ is a simple graph on $\pi$. The matrix $\mathscr{L}(h)$ is singular, and if $h$ is connected, then the algebraic multiplicity of the eigenvalue $0$ in $\mathscr{L}(h)$ is equal to $1$, i.e., the nullity of $\mathscr{L}(h)$ is 1.
\end{corollary}

\begin{corollary}\label{nulidad grafo externo}
If $g=\displaystyle{\bigvee_h a}$, where $a=(g_{\ell_j})^k_{j=1}$ is an ordered assembly of simple graphs, $\pi=(\ell_j)^k_{j=1}$ is a segmented partition of the linear order $\ell$ and $h$ is a simple graph on $\pi$, then the matrices $\mathscr{L}(h)$ and $L(h)$ have the same nullity.
\end{corollary}

The following proposition can be checked directly from the definition of the Laplacian matrix of a graph.
\begin{proposition}\label{autovector principal}
    Let $g$ be a simple graph with $n$ vertices.
\begin{enumerate}
    \item $0$ is an eigenvalue of $L(g)$ and $L(g^c)$ with eigenvector $1_{n\times 1}$.
    \item If $\lambda$ is an eigenvalue of $L(g)$ with eigenvector $\mathrm{v}\neq 1_{n\times 1}$, then $n-\lambda$ is an eigenvalue of $L(g^c)$ with the same eigenvector $\mathrm{v}$.
\end{enumerate}
\end{proposition}

For the Laplacian case, consider $\mathscr{T}\in \mathscr{F}_{\mathscr{G}_c}[\ell]$. From Corollaries \ref{nulidad grafo externo}, \ref{nulidad laplaciana}, equation (\ref{lapaciano iterado}) and Proposition \ref{autovector principal} it follows that

\begin{equation}
    \sigma(L(g^c))=0\cdot \left(\prod_{w\in \mathrm{Iv}(\mathscr{T})\atop w\neq r}\varphi_{_{n-N_{rw}}}(-\sigma(\mathscr{L}(g_w)))\right)\cdot \varphi_n(-\sigma(\mathscr{L}(g_r))).
\end{equation}

\section{The generalized characteristic polynomial of a graph and the operad $(\mathscr{G}_+, \bigvee)$}
The generalized characteristic polynomial of a graph $g$ is a bivariate polynomial $\phi_g(x, t)$ defined in \cite{LibroCvetkovic} by:
\begin{equation}\label{polinomiogeneral}
    \phi_g(x,t)=\det(xI-(A(g)-tD(g))).
\end{equation}

For certain values in $x, t$, the polynomial $\phi_g(x,t)$ generalizes the characteristic polynomials of $A(g)$ and $L(g)$. That is, if $\chi(B,x)$ denotes the characteristic polynomial of a square matrix in the variable $x$, then
\begin{eqnarray*}
    \chi(A(g), x)&=&\phi_g(x,0),\\
    \chi(L(g), x)&=&(-1)^n\phi_g(-x,1).
\end{eqnarray*}
If $a$ is an assembly of graphs, from equation (\ref{polinomiogeneral}) it follows that
\[\phi_a(x, t)=\prod_{g\in a}\phi_g(x,t).\]
As in Proposition \ref{Matriz join}, consider $g=\displaystyle{\bigvee_h a}$ where $a=(g_{\ell_j})^k_{j=1}$ is an assembly of $r_j$-regular graphs, $n_j=|\ell_j|$, $h$ a simple graph on $\pi=(\ell_j)^k_{j=1}$ and $\displaystyle{N_i=\sum_{\{\ell_s, \ell_i\}\in E(h)} n_s}$. Denote by $D(a,h)$ the diagonal matrix of order $k\times k$ given by $D(a, h)_{ij}=\delta_{ij}(r_i+N_i)$. Denote by $\phi_{(a,h)}(x,t)$ the characteristic polynomial in the variable $x$ corresponding to the matrix $A(a,h)-tD(a,h)$, that is
\[\phi_{(a,h)}(x,t)=\chi(A(a,h)-tD(a,h), x).\]
Note that if $a$ is an assembly of singular graphs, then $\phi_{(a,h)}(x,t)=\phi_h(x,t)$ and if $h$ is an assembly of singular graphs then $\phi_{(a, h)}(x,t)=\prod^k_{j=1}(x-r_j+tr_j)$.
Note that $1_{n_j\times 1}$ is an eigenvector of the matrix $A(g_{\ell_j})-tD(g_{\ell_j})$ with corresponding eigenvalue $r_j-tr_j$ and consequently the polynomial $x-(r_j-tr_j)=x-r_j+tr_j$ divides $\phi_{g_{\ell_j}}(x,t)$. Assuming that $a$ is an assembly of connected graphs and $h$ is a connected graph, Yu Chen and Haiyan Chen in \cite{YuHaiyanchen2019} obtain as a main result the equality
\begin{equation}\label{Chenfactor}
    \phi_g(x,t)=\phi_{(a,h)}(x,t)\prod^k_{j=1}\dfrac{\phi_{g_{\ell_j}}(x+tN_j,t)}{x-r_j+t(r_j+N_j)}.
\end{equation}

This equation shows that divisibility between assemblies of connected regular graphs, as in Definition \ref{divisibilidad}, is preserved in the context of generalized characteristic polynomials. Specifically, if $a_1$ and $a_2$ are assemblies of connected graphs such that the components of $a_1$ are regular graphs and $a_1$ divides $a_2$ ($a_1$ is a refinement of $a_2$), then each connected component of $a_2$ is $g^h_2=\displaystyle{\bigvee_{h} a^h_1}$ where $h$ is a connected component of the quotient graph $a_2/a_1$ and $a^h_1$ is the subassembly of connected graphs of $a_1$ with underlying partition equal to $V(h)$. Due to equation (\ref{Chenfactor}) each polynomial $\phi_{(a^h_1,h)}(x, t)$ divides $\phi_{a_2}(x,t)$ and therefore $\phi_{(a_1, a_2/a_1)}(x,t)$ divides $\phi_{a_2}(x,t)$.
Indeed, if $a_1=(g_{\ell_j})^k_{j=1}$, then
\begin{eqnarray*}
    \phi_{a_2}(x, t)&=&\prod_{h\in a_2/a_1} \phi_{g^h_2}(x,t)\\
    &=&\prod_{h\in a_2/a_1} \phi_{(a^h_1, h)}(x,t)\prod_{\ell_j\in V(h)}\dfrac{\phi_{g_{\ell_j}}(x+tN_j,t)}{x-r_j+t(r_j+N_j)}\\
    &=&\phi_{(a_1, a_2/a_1)}(x,t)\prod^k_{j=1}\dfrac{\phi_{g_{\ell_j}}(x+tN_j,t)}{x-r_j+t(r_j+N_j)}
\end{eqnarray*}

From the previous equality we get:
\begin{proposition}
Let $a_1, a_2, a_3$ be assemblies of connected graphs such that the components of $a_1$ and $a_2$ are regular. If $a_1$ divides $a_2$ and $a_2$ divides $a_3$, then
    $\phi_{(a_2, a_3/a_2)}(x,t)$ divides $\phi_{(a_1, a_3/a_1)}(x,t)$.
\end{proposition}
\begin{proof}
   \begin{eqnarray*}
        \phi_{a_3}(x,t)&=&\phi_{(a_2, a_3/a_2)}(x, t)\prod_{h\in a_2}\dfrac{\phi_h(x+tN_h,t)}{x-r_h+t(r_h+N_h)}\\
        &=&\phi_{(a_2, a_3/a_2)}(x, t)\prod_{h\in a_2}\dfrac{\phi_{(a^h_1, h/a^h_1)}(x+tN_h,t)}{x-r_h+t(r_h+N_h)}\prod_{g\in a^h_1}\dfrac{\phi_g(x+tN_h+tN_g,t)}{x-r_g+t(r_g+N_h+N_g)}\\
\end{eqnarray*}
In the equations above, $h\in a_2$ is a connected graph, $N_h$ is the external regularity of $h$ with respect to the quotient $a_3/a_2$. $a^h_1$ is the subassembly of $a_1$ whose vertex set is $V(h)$. $N_g$ is the external regularity of the connected graph $g\in a^h_1$ with respect to the quotient $h/a^h_1$, and $r_g, r_h$ are the respective regularities of the graphs $h$ and $g$. On the other hand
\begin{eqnarray*}
    \phi_{a_3}(x,t)&=&\phi_{(a_1, a_3/a_1)}(x, t)\prod_{g\in a_1}\dfrac{\phi_g(x+tN'_g,t)}{x-r_g+t(r_g+N'_g)},
\end{eqnarray*}
where $N'_g$ is the external regularity of the connected graph $g\in a_1$ with respect to the quotient $a_3/a_1$. Because the product $\bigvee$ is associative and satisfies the cancellation law, it follows that $N'_g=N_h+N_g$. Finally,
\begin{eqnarray*}
    \phi_{(a_1, a_3/a_1)}(x, t)&=&\phi_{(a_2, a_3/a_2)}(x, t)\prod_{h\in a_2} \dfrac{\phi_{(a^h_1, h/a^h_1)}(x+tN_h,t)}{x-r_h+t(r_h+N_h)}.
\end{eqnarray*}
\end{proof}
In \cite{LibroMiguel2015} it is shown that divisibility between assemblies of elements of an operad defines partially ordered sets; in particular the set of assemblies of simple graphs with vertices in $\ell$, $\mathscr{G}_+[\ell]$ with the divisibility relation is a poset. Moreover, if $a_3\in \mathscr{G}_+[\ell]$, then the set of quotient graphs $\mathcal{P}(a_3)=\{a_3/a : a \text{ divides } a_3, a \text{ is regular}\}$ together with the binary relation
$a_3/a_2 \leq a_3/a_1$ whenever $a_1$ divides $a_2$ defines a partial order on $\mathcal{P}(a_3)$. The polynomial ring modulo associates, $\mathbb{R}[x,t]/\sim$, together with divisibility, is a poset. In this context, if $\phi: \mathcal{P}(a_3)\longrightarrow \mathbb{R}[x,t]/\sim$ is defined by $\phi(a_3/a)=\phi_{(a, a_3/a)}(x,t)$, then the previous proposition states that $\phi$ is a morphism of posets.

If we iterate the factors on the right-hand side of equation (\ref{Chenfactor}), we can obtain a generalization of this equation using the notation and terminology presented previously, where the iterations of the product $\bigvee$ are described using the species of Schröder trees $\mathscr{F}_{\mathscr{G}_c}$. Indeed, if $g=\bigvee \mathscr{T}$, where $\mathscr{T}\in \mathscr{F}_{\mathscr{G}_c}[\ell]$, $r$ is the root of $\mathscr{T}$, and for each $w\in (\mathrm{Iv}(\mathscr{T})-\{v\})$ the graph $g_{\ell_w}$ has regularity $r_w$, then
\begin{equation}\label{polinomio generalizado iterado}
    \phi_g(x,t)=\left(\prod_{w\in \mathrm{Iv}(\mathscr{T})\atop w\neq r} \dfrac{\phi_{(a_w, g_w)}(x+tN_{rw},t)}{x-r_w+t(r_w+N_{rw})}\right)\phi_{(a_r, g_r)}(x,t).
\end{equation}
The proof of this equality follows in the same spirit as the proofs of Theorems \ref{adjacency spectrum iterated} and \ref{Laplacian spectrum iterated}.
Following the terminology of the theorem (\ref{colored spectrum}), given an admissible coloration $c: E(g)\longrightarrow \{1,2,...,k\}$ with the property that for $i\geq 2$, $g(i)$ is the assembly of connected regular graphs, using the coloration $c$, the equation (\ref{polinomio generalizado iterado}) is equivalent to
\begin{equation}\label{polinomio generalizado coloreado}
    \phi_g(x,t)=\left(\prod^k_{i=2}\prod_{h\in a_c(i)} \dfrac{\phi_{(g_h(i+1), h)}(x+tN_h,t)}{x-r_i+t(r_i+N_h)}\right) \phi_{(g(2), g/g(2))}(x,t).
\end{equation}

\section{The universal characteristic polynomial of a graph and the operad $(\mathscr{G}_+, \bigvee)$}

In \cite{Saravanan2021}, the universal adjacency matrix of a graph $g$ is defined as
\[U(g)=\alpha A(g)+ \beta I + \gamma J+ \delta D(g)\]
where \( \alpha, \beta, \gamma, \delta \) are real values with \( \alpha \neq 0 \), \( I \) is the identity matrix, and \( J \) is the matrix with all entries equal to 1, both of the same order as \( A(g) \). Note that the adjacency matrix \( A(g) \), the Laplacian matrix \( L(g) \), the signless Laplacian matrix \( Q(g) = D(g) + A(g) \), and the Seidel matrix \( S(g) = J - I - 2A(g) \) are obtained from \( U(g) \) for appropriate values of \( \alpha, \beta, \gamma, \delta \).

Let $\mathbb{C}(x)$ be the field of rational functions with coefficients in $\mathbb{C}$. Note that every square matrix $M$ has a non-zero characteristic polynomial, and therefore $xI-M$ is invertible over the field $\mathbb{C}(x)$. Let $u$ and $v$ be vectors with complex entries and the same number of rows as $M$. In \cite{Saravanan2021}, the \textit{main function} associated with the matrix $M$ corresponding to the vectors $u$, $v$ is defined as $\Gamma_M(u, v, x)=v^t(xI-M)^{-1}u$. When $u$ and $v$ are equal to the vector with all entries equal to 1 and $M$ is the universal adjacency matrix of a graph $g$, this function will be denoted by $\Gamma_g(x)$ for simplicity.

Consider as before $(a, h)\in \mathscr{G}_+(\mathscr{G}_+)[\ell]$ with $a=(g_{\ell_i})^k_{i=1}$, $\displaystyle{g=\bigvee_{h}a}$, $n_i=|\ell_i|$, $\displaystyle{N_i=\sum_{\{\ell_i, \ell_s\}\in E(h)} n_s}$, we define the \emph{main diagonal matrix}, $D_{\Gamma}(a, h)(x)$ as the diagonal matrix of order $k\times k$ that has in the entry $ii$ the function $\frac{1}{\Gamma_{g_{\ell_i}}(x-\delta N_i)} $, we also define the universal matrix of the factorization $(a, h)$ of $g$ as \[U_{\Gamma}(a, h)(x)=-\alpha A(h)+ \gamma I_k-\gamma J_k+ D_{\Gamma}(a, h)(x)\] The characteristic polynomial of the matrix $U(g)$ will be denoted by $\Phi_g(x)$, in \cite{Saravanan2021} they present a main result that generalizes Fiedler's lemma and from this they obtain that
\begin{equation}\label{polinomialuni}
\Phi_g(x)=\left(\prod^k_{i=1}\Phi_{g_{\ell_i}}(x-\delta N_i)\Gamma_{g_{\ell_i}}(x-\delta N_i)\right)det(U_{\Gamma}(a, h)(x))
\end{equation}

Assuming that $a$ is an assembly of regular graphs, in \cite{Saravanan2021} it is also obtained as a result that $p_i=\alpha\,\mathrm{reg}(g_{\ell_i})+\beta+\gamma n_i+\delta(\mathrm{reg}(g_{\ell_i})+N_i)$ is an eigenvalue of $U(g_{\ell_i})+\delta N_i I_{n_i}$. If $U(a, h)$ is the matrix of order $k\times k$ defined by \[(U(a, h))_{ij}=\sqrt{n_in_j}(\alpha A(h)_{ij}+\gamma(1-\delta_{ij}))+ \delta_{ij}p_i\] where $\delta_{ij}$ is the Kronecker delta and $\mu_{\delta N_i}(\sigma(U(g_{\ell_i})))$ is the word that results from degrading the multiplicity of $p_i-\delta N_i$ by one when this applies and then adds $\delta N_i$ to each variable of $\sigma(U(g_{\ell_i}))$. Then
\begin{equation}\label{spectrauni}
\sigma(U(g))=\left(\prod^k_{i=1}\mu_{\delta N_i}(\sigma(U(g_{\ell_i})))\right)\sigma(U(a,h))
\end{equation}

These equations tell us that graph factorizations using the operad $(\mathscr{G}_+, \bigvee)$ also factor the characteristic polynomials and spectra of universal matrices. Note that when $n_i=1$, it follows that $\Phi_{g_{\ell_i}}\Gamma_{g_{\ell_i}}=1$, and if $a$ is an assembly of singular graphs, it follows that $U_{\Gamma}(a, h)(x)=xI_n-U(g)$ and $U(a,h)=U(g)$. Let $\mathscr{T}\in \mathscr{F}_{\mathscr{G}_+}[\ell]$ be a factorization of a graph $g$, that is, $\displaystyle{\bigvee \mathscr{T}}=g$. Following reasoning similar to the proof of Theorem \ref{adjacency spectrum iterated}, it follows that by iterating equation (\ref{polinomialuni}), we obtain:
\begin{equation}\label{polinomio universal iterado}
    \Phi_g(x)=\left(\prod_{w\in \mathrm{Iv}(\mathscr{T})\atop w\neq r} \Gamma_{g_{\ell_w}}(x-\delta N_{rw})\,det(U_{\Gamma}(a_w, g_w)(x-\delta N_{rw}))\right)det(U_{\Gamma}(a_r, g_r)(x))
\end{equation}
where $r$ is the root of $\mathscr{T}$.

Alternatively, if $c:E(g)\longrightarrow \{1, 2,..., k\}$ is an admissible coloring, according to the terminology presented in the Theorem (\ref{colored spectrum}) except that $g(i)$ is not necessarily an assembly of regular graphs, the equation (\ref{polinomio universal iterado}) is rewritten as

\begin{equation}\label{universal colored polynomial}
 \Phi_g(x)=\left(\prod^k_{i=2} \prod_{h\in a_c(i)}\Gamma_{g_h}(x-\delta N_h) \det(U_{\Gamma}(g_h(i+1), h)(x-\delta N_h))\right)\det(U_{\Gamma}(g(2), g/g(2))(x))
\end{equation}

Iteration of equation (\ref{spectrauni}) is possible if we assume that for every $w\in (\mathrm{Iv}(\mathscr{T})-\{r\})$, the graph $g_{\ell_w}$ is regular and $q_w=(\alpha+\delta)\mathrm{reg}(g_{\ell_w})+\beta+\gamma |\ell_w|$ is a simple eigenvalue of $U(g_{\ell_w})$. In fact, we have the following proposition.

\begin{proposition}\label{autprin} If $g$ is a regular graph with $n$ vertices and $q=(\alpha+\delta)\mathrm{reg}(g)+\beta+\gamma n$, then $q$ is an eigenvalue of $U(g)$. Moreover, if $\alpha \gamma >0$ or $\gamma=0$ with $g$ connected, then $q$ is a simple eigenvalue of $U(g)$.
\end{proposition}

\begin{proof}
The quadratic form of $U(g)$ is given by $y^t U(g) y$, where $y=(y_1, y_2,..., y_n)$ is the vector of variables. From the definition of $U(g)$, it follows that
\[y^t U(g) y=\alpha\left(\sum_{\{i,j\}\in E(g)}2y_iy_j\right)+\delta\left(\sum_{\{i, j\}\in E(g)}y^2_i+y^2_j\right)+\gamma\left(\sum^n_{i=1}y_i\right)^2 +\beta \left(\sum^n_{i=1} y^2_i\right)\]
\[y^t qI_n y = \alpha\left(\sum_{\{i, j\}\in E(g)} y^2_i + y^2_j\right) + \beta \left(\sum^n_{i=1} y^2_i\right) +\gamma n \left(\sum^n_{i=1} y^2_i\right) + \delta \left(\sum_{\{i, j\}\in E(g)} y^2_i + y^2_j\right)\]
Subtracting both equations yields
\[y^t (qI_n-U(g)) y=\alpha\left(\sum_{\{i, j\}\in E(g)} (y_i - y_j)^2\right) + \gamma\left(\sum_{\{i, j\}\in E(K_n)} (y_i - y_j)^2\right)\]
where $K_n$ is the complete graph with $n$ vertices. From this equation and the assumptions about $\alpha, \gamma$, and $g$, it follows that $y^t(qI_n-U(g))y=0$ if and only if $y_1=y_2=\cdots =y_n$. Since $U(g)$ is symmetric, it is diagonalizable, and therefore $q$ is an eigenvalue of $U(g)$ with geometric multiplicity equal to its algebraic multiplicity, which in this case is 1. That is, $q$ is a simple eigenvalue of $U(g)$.
\end{proof}

Again consider the factorization $g=\bigvee \mathscr{T}$. If for each $w\in (\mathrm{Iv}(\mathscr{T})-\{r\})$ the graph $g_{\ell_w}$ satisfies the premises of Proposition \ref{autprin}, then following reasoning analogous to the proof of Theorem \ref{Laplacian spectrum iterated}, it is verified that by iterating equation (\ref{spectrauni}) we obtain
\begin{equation}
\sigma(U(g))=\left(\prod_{w\in \mathrm{Iv}(\mathscr{T})\atop w\neq r} \mu_{\delta N_{rw}}(\sigma(U(a_w, g_w)))\right)\sigma(U(a_r, g_r))
\end{equation}
where $r$ is the root of the Schröder tree $\mathscr{T}$ and $\mu_{\delta N_{rw}}(\sigma(U(a_w, g_w)))$ is the word or monomial resulting from decreasing the multiplicity of $q_w=(\alpha+\delta)\mathrm{reg}(g_{\ell_w})+\beta+\gamma |\ell_w|$ by one when applicable and then adding the value $\delta N_{rw}$ to each variable of $\sigma(U(a_w, g_w))$.

If $c:E(g)\longrightarrow \{1, 2,..., k\}$ is an admissible coloring such that for each $i\geq2$ and $h\in a_c(i)$ the graph $g_h=\displaystyle{\bigvee_h g_h(i+1)}$ satisfies the conditions of Proposition \ref{autprin}, then

\begin{equation}
\sigma(U(g))=\left(\prod^k_{i=2}\prod_{h\in a_c(i)} \mu_{\delta N_h}(\sigma(U(g_h(i+1), h)))\right)\sigma(U(g(2), g/g(2)))
\end{equation}

In future work, we will examine whether the structure of other examples of operads defined on graphs determines factorizations of matrices, polynomials, and the spectra of these matrices, provided these matrices depend on the operadic product. Undoubtedly, Schröder trees are the combinatorial structure that describes generalized factorizations. The connection between spectral graph theory and the theory of operads is fundamental for the development of research in mathematics and its related areas, which contributes to the training of future scientists, as well as higher education students.

\bibliographystyle{plain}
\bibliography{graph}

\end{document}